\documentclass[12pt]{amsart}
\usepackage[cp1250]{inputenc}
\usepackage[T1]{fontenc}
\usepackage{amscd}

\usepackage{amsmath}
\usepackage{amsthm}
\usepackage{amssymb}
\usepackage{amscd}
\usepackage{graphicx}
\usepackage{epsfig}
\usepackage{bbm}

\theoremstyle{plain}
\newtheorem{thm}{Theorem}[section]
\newtheorem{lem}[thm]{Lemma}
\newtheorem{prop}[thm]{Proposition}
\newtheorem{cor}[thm]{Corollary}

\theoremstyle{definition}
\newtheorem{dff}[thm]{Definition}
\newtheorem{rem}[thm]{Remark}
\newtheorem{exa}[thm]{Example}

\numberwithin{equation}{section}
\def\tt{\mathrm {\bf t}}
\def\ss{\mathrm {\bf s}}
\def\mm{\mathrm m}
\def\ii{\mathrm i}
\def\i{\iota}
\def\we{\wedge}

\def\f{\varphi}
\def\fs{\sigma}

\def\p{\partial}
\def\a{\alpha}
\def\b{\beta}
\def\o{\omega}

\def\La{\Lambda}
\def\F{\mathcal F}

\def\cc{\CC^{\infty}}

\def\ci{\circ}

\def\rz{\mathbb{R}}
\def\R{\rz}

\def\s{\subset}
\def\ddt{{d\over d t}}
\def\mt{\mapsto}
\def\cm{\CC^{\infty}(M,\mathbb R)}
\def\ccm{\CC^{\infty}_c(M,\mathbb R)}

\def\cmm{\CC^{\infty}(I\t M,\mathbb R)}
\def\ccmm{\CC^{\infty}_c(I\t M,\mathbb R)}

\def\t{\times}
\def\r{\rightarrow}
\def\rr{\rightrightarrows}
\def\cc{\CC^{\infty}}

\def\pp{\partial}
\def\pl{\parallel}

\def\plaa{\pl_{\infty}}
\def\la{\lambda}
\def\gg{\Gamma}

\def\ra{\varrho_{H}}

\def\rb{\ra}

\def\rh{\varrho_H}

\def\la{\lambda}

 \DeclareMathOperator{\sect}{Sect}
\DeclareMathOperator{\graph}{graph}
 \DeclareMathOperator{\symp}{Symp}

 \DeclareMathOperator{\diff}{Diff}
 
 \DeclareMathOperator{\ham}{Ham}
 \DeclareMathOperator{\id}{id}
\DeclareMathOperator{\fl}{Fl}

\DeclareMathOperator{\CC}{C} 
\DeclareMathOperator{\length}{length}

\def\m{(M,\Lambda)}
\def\LL{\Lambda}
\def\lea{\length}

\def\l{\length}

\def\fll{\mathcal F_{\Lambda}}

\keywords{Poisson manifold, Poisson map, Hamiltonian diffeomorphism, Hamiltonian group,  Hofer metric,
 non-degeneracy, symplectic groupoid, integrability of Poisson manifolds} 
\address{Faculty of Applied Mathematics, AGH University of Science and
\linebreak Technology, al. Mickiewicza 30, 30-059 Krak\'ow,
Poland} \email{tomasz@agh.edu.pl}
\date{May 23, 2016}

\title{ On the existence   of a Hofer type metric for Poisson manifolds}

\author{ Tomasz Rybicki}

\begin{document}

\maketitle

\begin{abstract} An analogue  of the Hofer metric $\ra$ on the Hamiltonian group $\ham\m$ of a Poisson manifold $\m$ can be defined but there is the
problem of its non-degeneracy. First we observe that $\ra$ is a genuine metric on $\ham\m$  when  the union of all proper leaves of the corresponding symplectic foliation 
is dense.
Next we deal with the important class of integrable Poisson manifolds. Recall that a Poisson manifold is called integrable if it can be realized as the space 
of units of a symplectic groupoid. Our main result states that $\ra$ is a Hofer type metric for every  Poisson manifold which admits a Hausdorff integration.

\end{abstract}

\section{Introduction}

 Given any symplectic manifold $(N,\o)$, the Hofer metric $\rh$ is a bi-invariant metric on the group of
compactly supported Hamiltonian symplectomorphisms  $\ham(N,\o)$ of  $(N,\o)$.  Hofer geometry constitutes a basic tool in
symplectic topology (see \cite{HZ}, \cite{MS}, \cite{MS1},
\cite{Pol1}). Our aim is to generalize the Hofer metric to the case of Poisson manifolds, a non-transitive counterpart of symplectic manifolds.
It is a rather easy observation that there exists a possibly degenerate analogue of the Hofer metric $\ra$ on the Hamiltonian group $\ham\m$ of an arbitrary Poisson manifold $\m$.
In the  case of many important types of Poisson manifolds $\ra$ occurs to be a genuine bi-invariant metric (see Theorems 1.1 and 1.3). 
Thus, the Hofer geometry could be considered in a much more general setting.

Let $(M, \La)$ be a Poisson manifold,
i.e. $M$ is a  smooth paracompact manifold endowed with a bivector $\La$, called the \emph{Poisson bivector},  which satisfies \begin{equation}[\Lambda ,\Lambda ]=0,
\end{equation} where
$[\cdot,\cdot]$ is the Schouten-Nijenhuis bracket (cf.\cite{Vai}). The bivector $\La$ induces a 'musical' bundle  homomorphism 
$  \sharp :T^*M\r TM$ given by $ \beta_x (\sharp\alpha_x )=\La_x (\alpha_x ,\beta_x )$ for all $x\in M$ and all 1-forms $\a, \b$.
The image of $\sharp $ integrates to a generalized foliation, denoted by $\F_{\Lambda}$ and called the \emph{symplectic} foliation.
 The homomorphism $\sharp$ allows us to define the group of compactly supported  Hamiltonian diffeomorphisms $\ham\m$
as well (see section 2.2).

For any $F\in\CC^\infty_c(M,\R)$ denote \begin{equation} \pl
F\plaa=\max_{p\in M}F(p)-\min_{p\in M}F(p).\end{equation}
Let $\mathcal P\ham\m$ be the set of Hamiltonian isotopies of $\m$ starting at the identity (see section 2.2). Since the Lie algebra $\ccm$ contains Casimir elements,
 $\mathcal P\ham\m$ will be regarded as the set of all pairs
\begin{equation}
\Phi_F=(\{\phi_F^t\},F),\quad F\in\ccmm,
\end{equation}  where  $I=[0,1]$ and the  isotopy $\{\phi_F^t\}$ is generated by a Hamiltonian $F\in\ccmm$. In particular, $\phi_F^0=\id$. Next,  by using (1.2) we define the length
of $\Phi_F$
\begin{equation}
\l(\Phi_F)=\int_0^1\pl F(t,\cdot)\plaa dt.
\end{equation}
 For $F\in\ccmm$ the symbol  $F\mt\phi$ will denote that the isotopy  $\{\phi_F^t\}$ satisfies the condition $\phi_F^1=\phi$.
Then for  $\phi\in\ham\m$ we set
\begin{equation}\rb(\phi,\id)=\inf\{\l(\Phi_F)|\; F\in\ccmm,\; F\mt\phi\},
\end{equation}
Here $\Phi_F$ is defined by (1.3) and $\l(\Phi_F)$ by (1.4).


\begin{thm}  Given a Poisson manifold $(M,\La)$, the formula (1.5) defines a bi-invariant pseudo-metric $\ra$ on the group of Hamiltonian diffeomorphisms $\ham\m$ of $\m$.
Moreover, $\ra$ is a metric whenever the union of all proper  leaves of  the symplectic foliation $\F_{\Lambda}$ of $\m$ is dense.
\end{thm}
Here the properness of a leaf $L$ means that the leaf topology of $L$ coincides with the induced topology from $M$ and that the dimension of $L$ is positive.

Recently, Sun and Zhang \cite{SZ} claimed  the non-degeneracy of $\ra$ for all regular Poisson manifolds. However, their proof contains an essential error (see 
Remark 3.2).

Denote by $\sigma_L$ the symplectic form living on a leaf $L\in\F_{\Lambda}$, and by $c_L$ the Gromov width of the symplectic manifold $(L,\sigma_L)$.
Then we have the following generalization of the energy-capacity inequality.
\begin{cor}
 Let $\m$ be an arbitrary Poisson manifold.    For any open ball $U\s M$ 
let $c_{\La}(U)$ denote  $\sup\{c_L(U\cap L)\}$, where $L$ runs over all proper leaves of $\F_{\Lambda}$ such that $U\cap L$ is an open ball in $L$ (and we 
put $c_{\La}(U)=0$ if there is no leaf $L$ satisfying the above condition). Then we have
\begin{equation}
 \frac{1}{2}c_{\La}(U)\leq E(U),\end{equation} where $E(U)$
 is the displacement energy of $U$ defined by means of $\ra$, see section 3.3. Consequently, if $\m$ is such that the union of all proper  leaves of  the symplectic foliation $\F_{\Lambda}$  is dense in $M$, then $E(U)>0$ for all open sets $U$.

\end{cor}

Our main result is the following

\begin{thm}
Suppose that $\m$ is a Poisson manifold which is integrated by a Hausdorff symplectic groupoid. Then the bi-invariant pseudo-metric $\ra$ is a genuine metric.
\end{thm}

\begin{rem}
The proof of the non-degeneracy of the Hofer metric is hard, and \emph{a fortiori} the problem of whether $\rh$ is degenerate or non-degenerate is possibly even harder for Poisson manifolds.
This is the reason that I cannot provide any example of a possible degeneracy of $\rh$.
\end{rem}

In section 5 we shall give  examples of Poisson manifolds integrated by a Hausdorff symplectic groupoid. The significance of the integrability follows from the fact that the existence of a groupoid
integrating $\m$ can be thought of as a desingularization of it (cf. the paragraph before Remark 6 in \cite{CF2}). A program of using symplectic groupoids
in the quantization theory was formulated and partially carried out in the papers by Weinstein \cite{We2},
Weinstein and Xu \cite{WX}, Zakrzewski \cite{Za} and in some other papers.  The program may be  described in the following way: "if a Poisson manifold $\m$ is seen 
as the phase space of a mechanical system, the symplectic groupoid $(\gg,\sigma)$ that integrates $\m$ (if it exists) can also be interpreted  as a phase space of the system, which has more coordinates than
 really needed but, instead it has some kind of \emph{symmetries}, the left and right translations of the groupoid. Hence, it is natural to look at the quantization of $(\gg,\sigma)$ as the corresponding quantum system, and to study whether it  also has similar symmetries", cf. \cite{Vai}, p.158.  Thus, the integrability property is very important and natural in the category of Poisson manifolds and the Hofer metric might be an important tool
in Poisson geometry.

It is a striking fact that the "symplectic" proof of Theorem 1.3 (section 4.2) makes use essentially of the integrability property. On the other hand, bearing in mind the 
 hard symplectic methods in the proof of non-degeneracy of $\ra$ in the symplectic case (cf. \cite{LM} and references therein), one could hardly expect that a specific "Poisson" proof,
possibly for all Poisson manifolds, is available.

Jacobi manifolds constitute an important generalization of Poisson manifolds, cf.\cite{DLM}. A general problem is to introduce a Hofer type metric on the group of strict Hamiltonian
diffeomorphisms of a Jacobi manifold. More precisely, it is straightforward to introduce a bi-invariant pseudo-metric, but there is the problem of its non-degeneracy. Every leaf of the generalized foliation induced by a Jacobi structure admits either a symplectic structure, or a locally conformal symplectic structure (which is not symplectic), or a contact structure. Banyaga and Donato introduced a bi-invariant metric of Hofer type on the strict contactomorphism group
of a special kind of contact manifold \cite{BD}. Recently, M\"uller and Spaeth extended this result to all contact manifolds \cite{MSp}. Locally conformal symplectic manifolds form an important generalization of symplectic geometry and many facts from this geometry can be carried over to the locally conformal symplectic case (\cite{HR1}, \cite{HR2}, \cite{BK}). However, it seems that a possible proof of  the non-degeneracy of the pseudo-metric in this case is hard, since it cannot be established by appealing to the symplectic energy-capacity inequality like in \cite{MSp} or in the present paper.

\medskip

{\bf Acknowledgments.} I would like to express my gratitude to the referees for critical comments and valuable remarks which enable me to improve essentially the first version
of the paper.

\section{The groups of automorphisms of a Poisson manifold}

\subsection{Poisson manifolds} 
Recall that a Poisson structure on a smooth manifold $M$, introduced first by  Lichnerowicz,  can be defined by a bivector
 $\Lambda $ such that (1.1) is fulfilled.  Then the rank of $\La _x$
may vary but it is even everywhere.
 We have the 'musical' bundle  homomorphism $\sharp$ associated with $\La $ defined for any $x\in M$ by
$$  \sharp=\sharp_{\La} :T^*M\r TM, \quad \beta_x(\sharp\alpha_x)=\La_x (\alpha_x ,\beta_x ),\quad  \forall \a,\beta\in \sect(T^*M).$$
In the case of $\La $  nondegenerate (i.e. the rank of $
 (\La_x )$ is equal to $
\dim(M)$),  we get
a symplectic structure $\o$, and $\sharp $ is an isomorphism.
The distribution $\sharp(T^*_xM),\ x\in M$,
integrates to a generalized foliation  such that $\La $ restricted to any
leaf induces a symplectic structure. This foliation is called \emph{symplectic}
and  denoted by $\F_{\La }$. If the mapping $\sharp_{\La}$ has  constant rank, then
 the Poisson structure $\La $ is called \emph{regular}.

Consider  a bivector field $\LL$  on a smooth manifold $M$. For all smooth functions $F,H\in\cm$ we set 
\begin{equation}
\{F,H\}=\i(\LL)(dF\we dH)=\LL(dF\we dH),
\end{equation}
where $\i$ is the interior product. Clearly the above bracket is 2-linear and antisymmetric. It is known that if $\La$ satisfies the equality (1.1), then  the bracket (2.1) fulfills the Jacobi identity
\begin{equation}
\{F,\{G,H\}\}+\{H,\{F,G\}\}+\{G,\{H,F\}\}=0,\quad F,G,H\in\cm,
\end{equation}
as well as the  Leibniz rule
$$\{F,GH\}=\{F,G\}H+G\{F,H \}.$$

Since $\{F,\cdot\}$ is a derivation of $\cm$ the following holds. For any $F\in\cm$ there is a well-defined vector field $X_F$ such that for all $H\in\cm$ we have
\begin{equation}
\{F,H\}=X_FH=-X_HF.
\end{equation}
Then in view of (2.2) and (2.3) we conclude that 
\begin{equation}
X_{\{F,H\}}=[X_F,X_H],
\end{equation}
that is $\cm\ni F\mt X_F\in\frak X(M)$ is a Lie algebra homomorphism. (Hereafter we follow the sign convention from Vaisman's monograph \cite{Vai}.)

\subsection{Automorphisms of a Poisson manifold}
For a smooth manifold $M$ by $\diff_c(M)_o$ we denote the compactly supported identity component of the group of all diffeomorphisms on $M$, and let $\mathcal P\diff_c(M)_o=
\{f\in\cc(\R,\diff_c(M)_o| f(0)=\id\}$ be the set of smooth isotopies. If $\frak X_c(M)$ is the Lie algebra of all compactly supported
vector fields on $M$, then we have the bijection
\begin{equation}\mathcal P\diff_c(M)_o\ni\{f_t\}\mt \{\dot f_t\}\in\cc(\R,\frak X_c(M)),\end{equation} where
for all $p\in M$ and $t\in\R$ we have
\begin{equation}
         \dot f_t(p)={\pp f_t\over \pp t}( f_t^{-1}(p)).
\end{equation}
In particular, a time-independent
vector field $X\in\frak X_c(M)$ corresponds to its flow
$\fl^X\in\mathcal P\mathcal\diff_c(M)_o$.

 Likewise, for any $\f\in\diff_c(M)_o$
the space $\mathcal P_\f\diff_c(M)_o$ of all smooth isotopies starting at $\f$ identifies with $\CC^{\infty}(\R,\frak X_c(M))$ by
\begin{equation}\mathcal P_\f\diff_c(M)_o\ni\{f_t\}\mt \{\dot{\overbrace{f_t\ci\f^{-1}}}\}\in\cc(\R,\frak X_c(M)).\end{equation}

The following is easy to check. \begin{lem} Let $\{f_t\},
\{g_t\}\in\mathcal P\diff_c(M)_o$ and $\phi\in\diff(M)$.
Then:
\begin{enumerate}
\item  $\dot{\overbrace{f_tg_t}}=\dot f_t+(f_t)_*(\dot g_t)$.
\item  $\dot{\overbrace{f_t^{-1}}}=-(f_t^{-1})_*(\dot f_t)$. \item
$\dot{\overbrace{\phi f_t\phi^{-1}}}=\phi_*(\dot f_t)$.
\end{enumerate}
\end{lem}

For  arbitrary Poisson manifolds $(M_1,\La_1)$ and $(M_2,\La_2)$
a smooth mapping $f$ of $(M_1,\La_1 )$ into $(M_2,\La_2)$ is called a \emph{Poisson
morphism} if
$$
\{F\circ f,H\circ f\}=\{F,H\}\circ f \quad \hbox {for any}\ F,H\in\cc(M_2,\R).
$$
In this case we have
\begin{equation}
f_*X_{F\ci f}=X_F,
\end{equation}
see, e.g., \cite{Vai}, p.97. The symbol $\diff(M,\La )$ stands for the group of all Poisson automorphisms
of $(M,\La )$.
Next, $\frak {ham}\m$ denotes the space of all (compactly supported) Hamiltonian vector fields, i.e. $X\in \frak{ham}\m$ iff there exists a compactly
supported function $F\in \ccm$
such that
$X=[\La ,F]$ or, equivalently, \begin{equation} X=\sharp(dF).\end{equation} Then we will write $X=X_F$ and this definition coincides with that given by (2.3).   Recall that
a vector field $X$ is an infinitesimal automorphism of $(M,\La )$ if $[\La ,X]=0$,
that is $L_X\La =0$, where $L$ is the Lie derivative. Let $\frak X\m$ stand for the space of all i.a. of $\m$.
Observe that for $Y\in \frak X(M,\La )$, $X_F\in \frak{ham}(M,\La )$ we get $[Y,X_F]= L_Y[\La ,F]=[\La , L_YF]$.
Consequently, $\frak{ham}(M,\La)$ is an ideal of $\frak X(M,\La)$.

 From now on we write $F_t=F(t,\cdot)$ for $F\in\cmm$. To any $F\in\ccmm$ one can assign the smooth path $\{t\mt X_{F_t}\}\in\cc(\R,\frak{ham}\m)$ by means of the homomorphism
$\sharp$. Next, by
using (2.5), one assigns to $F$ the isotopy denoted by $\{\phi_F^t\}$ as the unique element of $\mathcal P\diff_c(M)$ corresponding to $\{t\mt X_{F_t}\}$. Let $$\mathcal P\ham\m
=\{\Phi_F=(\{\phi_F^t\},F)|\;F\in\ccmm,\;\phi_F^0=\id\}$$ be the set of all Hamiltonian isotopies of $\m$ starting at the identity. Consequently we have the bijection
\begin{equation}\ccmm\ni F\mt \Phi_F\in\mathcal P\ham\m. \end{equation}
Similarly, for any $\varphi\in\diff\m$, the bijection (2.7) induces the bijection 
\begin{equation}\ccmm\ni F\mt \Phi_F\in\mathcal P_\varphi\ham\m, \end{equation}
 where $\mathcal P_\varphi\ham\m$ is the set of all Hamiltonian isotopies starting at $\varphi$.
  A diffeomorphism $\phi$ is called \emph{Hamiltonian} if there exists a Hamiltonian isotopy
$\Phi_F=(\{\phi_F^t\},F)$ such that $\phi_F^0=\id$ and $\phi_F^1=\phi$. In this situation we write $F\mt\phi$.

By  $\ham\m$ we denote the set of all Hamiltonian diffeomorphisms. Clearly $\ham\m\s\diff_c\m$.

Recall that for $F, H\in\ccmm$ the product $F\#H$, the inverse $\overline{F}$, and the pull-back $f^*F$, where $f\in\diff\m$, are given by the formulae 
\begin{align}
\begin{split}
(F\#H)_t&=F_t+H_t\ci(\phi_F^t)^{-1},\\
 (\overline F)_t&=-F_t\ci\phi_F^t,\\
(f^*F)_t&=F_t\ci f.\\
\end{split}
\end{align}

Now from Lemma 2.1 we derive the following
\begin{lem}
Under the above notation one has
\begin{enumerate}
\item $F\# H$ generates $\{\phi_F^t\ci\phi_H^t\}$ and
$\Phi_{F\#H}=\Phi_F\cdot\Phi_H:=(\{\phi_F^t\ci\phi_H^t\},F\#H)$; \item $\overline F$ generates $\{(\phi_F^t)^{-1}\}$ and $\Phi_{\overline F}=\Phi_F^{-1}:=(\{(\phi_F^t)^{-1}\},\overline F)$;
\item $\Phi_{{\overline F}\#H}=\Phi^{-1}_F\cdot\Phi_H$; \item $\Phi_{f^*F}=(\{f^{-1}\ci\phi_F^t\ci f\},f^*F)$.
\end{enumerate}
In particular, $\mathcal P\ham\m$ is a group.
\end{lem}

\begin{prop}
 $\ham\m$ is a normal subgroup of $\diff(M,\La )$.
\end{prop}

In fact, this follows immediately from Lemma 2.2 and the definition of Hamiltonian diffeomorphism (\cite{Ry}).

The length of $\Phi_F$ is given by (1.4).
It is important that a reparametrization does not change the isotopy length.
\begin{lem}
Let $\fs:[a,b]\to I$ be a smooth non-decreasing surjection. Then for
any Hamiltonian isotopy $\Phi_F=(\{\f^t\},F)$    one has:
\begin{enumerate}
\item the isotopy $\{\f^{\fs(t)}\}$ is Hamiltonian generated by $F^{\fs}(t,x)=\fs'(t)F(\fs(t),x)$;
\item
$\lea(\Phi_{F^{\fs}})=\lea(\Phi_F)$, if $\Phi_{F^{\fs}}=(\{\f^{\fs(t)}\}_{a\leq t\leq b},F^{\fs})$;
\item For any isotopy $\Phi_F=(\{\f^t_F\},F)$ and for any   $0<\delta<\frac{1}{2}$ there exists an isotopy
$\Phi_{\hat F}=(\{\hat \f^t\},\hat F)\in\mathcal P\ham\m$ for some $\hat F\in\ccmm$ with $\lea(\Phi_{\hat F})
=\lea(\Phi_F)$ such that for all $0\leq t\leq\delta$ one has $\hat\f^t=\f^0_F$  and
$\hat\f^{1-t}=\f^1_F$. 
\end{enumerate}
\end{lem}
\begin{proof}

 (1) Define $F^{\fs}\in\CC^{\infty}_c([a,b]\t M,\R)$ by $F^{\fs}(t,x)=\fs'(t)F(\fs(t),x)$. Then $F^{\fs}$ corresponds to  $\{\f^{\fs(t)}\}$ by (2.10).
In fact,
\begin{align*}\dot{\overbrace{\f^{\fs(t)}}}(x)&=\fs'(t)\frac{\p
\f^{\tau}}{\p \tau}|_{\tau=\sigma(t)}((\f^{\fs(t)})^{-1}(x))
=\fs'(t)\dot \f^{\fs(t)}(x)\\
& =\fs'(t)\sharp(d_{(x)}(F(\fs(t),x)))
=\sharp(d_{(x)}(\fs'(t)F(\fs(t),x)))\\
&=\sharp(d_{(x)}(F^{\fs}(t,x))),\end{align*}
where $d_{(x)}$ is the differential with respect to $x$.
Now (2) follows from (1) since
\begin{align*}\lea(\Phi_{F^{\fs}})&=\int_a^b\pl
F^{\fs}_t\plaa dt=\int_a^b\fs'(t)\pl
F_{\fs(t)}\plaa dt\\
&=\int_0^1\pl F_t\plaa dt=\lea(\Phi_F).\end{align*} Finally, (1) and (2) imply (3).
\end{proof}

As a consequence of Lemma 2.4 we have an equivalent definition of the Hamiltonian diffeomorphisms. Namely, denote
by $\hat{\mathcal P}_{\f}^{\psi}\ham\m$ the totality of isotopies $\Phi_F=(\{\f^t_F\},F)\in\mathcal P\ham\m$ such that \begin{equation} (\exists\delta>0),(\forall 0\leq t\leq\delta),\; \f_F^t=\f,\; \f_F^{1-t}=\psi.
\end{equation}
Then $\f\in\ham\m$ iff there is $F\in\ccmm$ such that $\Phi_F\in\hat{\mathcal P}_{\id}^\f\ham\m$.
It follows that the concatenation of isotopies can serve in the definition of the multiplication in $\ham\m$.

\section{Generalization of the Hofer metric for Poisson manifolds}

\subsection{A bi-invariant pseudo-metric}
Let $G$ be a group. For a function $\nu:G\r[0,\infty)$ such that
$\nu(e)=0$ consider the following conditions. For any $g,h\in G$
\begin{enumerate}  \item
$\nu(g^{-1})=\nu(g)$; \item $\nu(gh)\leq\nu(g)+\nu(h)$; \item
$\nu(g)>0$ if and only if $g\neq e$;  \item
$\nu(hgh^{-1})=\nu(g)$.
\end{enumerate}
Then $\nu$ is called a \emph{pseudo-norm} (resp. \emph{norm}) if
(1)-(2) (resp. (1)-(3)) are fulfilled. If (3) is not satisfied,
$\nu$ is called \emph{degenerate}. Next, $\nu$ is
\emph{conjugation-invariant} if (4) is satisfied. Conjugation-invariant pseudo-norms correspond bijectively to bi-invariant pseudo-metrics on $G$:
$\varrho(g,h)=\nu(gh^{-1})$ and conversely $\nu(g)=\varrho(g,e)$.

Observe that in condition (4) the element $h$ may belong to some larger group $\tilde G$ such that $G$ is a normal subgroup of $\tilde G$.

Given a Poisson manifold $\m$, by making use of formulae (1.2), (1.3) and (1.4) we define the pseudo-norm $E_H$ by
\begin{equation}E_H(\phi)=\ra(\phi,\id)=\inf\{\l(\Phi_F)|\; \Phi_F\in\mathcal P\ham\m,\; F\mt\phi\}.
\end{equation}
In view of (2.12) and Lemma 2.2 it is easily seen that $E_H$ defined by (3.1) satisfies conditions (1), (2) and (4) above.
Indeed, for all $\Phi_F,\Phi_H\in\mathcal P\ham\m$ we have
\begin{align}\begin{split}
\lea(\Phi_F\cdot\Phi_H)&\leq\lea(\Phi_F)+\lea(\Phi_H),\\
\lea(\Phi_F^{-1})&=\lea(\Phi_F),\\
  \lea(\Phi_{f^*F})&= \lea(\Phi_F), \quad \forall f\in\diff\m.
\end{split}\end{align}
A possible proof of (3) for $E_H$ is a hard problem.

Further, for all $\phi,\psi\in\ham\m$ we put
$$ \ra(\phi,\psi)=E_H(\phi\psi^{-1}).
$$
 In view of (2.11) and Lemma 2.4
 it is easily checked that  
$$ \ra(\phi,\psi)=\inf\{\l(\Phi_F)|\; \Phi_F\in\hat{\mathcal P}^\psi_\phi\ham\m\}, 
$$
that is $\ra$ is determined by isotopies satisfying (2.13).

Concerning the symplectic case the non-degeneracy condition (3) was proved by Hofer in \cite{Hof} for
$M=\R^{2n}$ (see also \cite{HZ}). It was generalized for some other symplectic
manifolds by Polterovich in \cite{Pol}. Finally, the proof  for
all symplectic manifolds was given by Lalonde and McDuff in
\cite{LM}. In all three proofs  Hofer's idea of displacement
energy and hard symplectic methods are in use.

\begin{thm} Let $(M,\o)$ be an arbitrary symplectic manifold. Then
$E_H:\ham(M,\omega)\r[0,\infty)$ is a conjugation-invariant norm, called the \emph{Hofer energy}. Consequently,
$\varrho_H(\f,\psi):=E_H(\f\psi^{-1})$ is a bi-invariant metric,
called the \emph{Hofer metric}.
\end{thm}

The Hofer metric plays a crucial role in symplectic topology and
various important notions and facts are expressed in terms of it
(see, e.g., \cite{HZ}, \cite{MS}, \cite{MS1}, \cite{Pol1}). The
original proof of its non-degeneracy for $M=\R^{2n}$ is based on
the action principle and a clue role in it is played by the
action spectrum. The Hofer metric is intimately related, on the
one hand, to a capacity $c_0$ (c.f. \cite{HZ}) and hence to
periodic orbits, and on the other hand to the displacement energy.
Of course, the proofs of Theorems 1.1 and 1.3 are based on Theorem 3.1 or on the energy-capacity inequality.

\subsection{Proof of Theorem 1.1}
In view of  (3.2) and a standard reasoning we know that $\ra$ is a bi-invariant pseudo-metric. To show the non-degeneracy of $\ra$, let $\phi\in\ham\m$, $\phi\not=\id$, and let $F\mt\phi$, where
$F\in\ccmm$. Then there exists a proper leaf $L\in\fll$ such that $\phi|_L\not=\id_L$. Denote by $F_L$ the restriction of $F$ to $I\t L$, i.e. $F_L(t,x)=F(t,x)$ for $(t,x)\in I\t L$.
Due to Weinstein's splitting theorem for Poisson manifolds or Dazord's splitting theorem for singular foliations (see, e.g., \cite{Vai}) we get that $F_L\in\CC_c^\infty(I\t L,\R)$.
It follows that $F_L\mt\phi|_L$ since the leaves of $\F_{\LL}$ are preserved by Hamiltonian isotopies. In view of Theorem 3.1 applied to $(L,\sigma_L)$ we have
\begin{equation} 0<\ra^L(\phi|_L,\id)\leq\int_0^1\pl F_L(t,\cdot)\plaa dt\leq \int_0^1\pl F(t,\cdot)\plaa dt,\end{equation}
where $\ra^L$ is the Hofer metric on $\ham(L,\sigma_L)$. Since $F$ is an arbitrary element of $\ccmm$ satisfying $F\mt\phi$, we get $\ra^L(\phi|_L,\id)\leq\ra(\phi,\id)$.
Thus $\ra$ is non-degenerate.

\begin{rem}
 Sun and Zhang in \cite{SZ} claimed that $\ra$ is non-degenerate for all \emph{regular} Poisson manifolds. They used a bit different definition of $\ra$ by using in it the Casimir
functions. In our paper we omit Casimirs in definition (3.1) and clearly this leads to the same $\ra$ (see our definition of Hamiltonian isotopies in 2.2). However, there is
an essential error in the proof of the non-degeneracy of $\ra$ in \cite{SZ}. Namely, in (28) in \cite{SZ} we do not know whether the restriction $\tilde h$ of
a Hamiltonian function $h$ is \emph{compactly supported} on a symplectic leaf $P_{\a_0}$. Consequently estimations (2.8) and (2.9) need not be true unless the properness
of leaves of the induced symplectic foliation is assumed. On the other hand, I do not know any example of a Poisson manifold such that $\rh$ is degenerate (Remark 1.4).

\end{rem}

\subsection{Proof of Corollary 1.2}
We define the displacement energy as in the symplectic case:
$$ E(U)=\inf\{E_H(\phi):\; \phi\in\ham\m,\; \phi(U)\cap U=\emptyset\}.
$$
If there is no $\phi$ that displaces $U$ we  put $E(U)=+\infty$. Observe that $E(U)\leq E(V)$ whenever $U\s V$, and $E(\phi(U))=E(U)$ for all $\phi\in\ham\m$.

 Let $L\in\F_{\La}$ be a proper leaf and let $E_L$ denote the displacement energy for the symplectic manifold $L$. Clearly we have
\begin{equation}
E_L(U\cap L)\leq E(U)
\end{equation}
for all open subsets $U$ of $M$, since $\phi|_L$ is a Hamiltonian diffeomorphism of $L$ (see the proof of 1.1) and  $\phi|_L$ displaces $U\cap L$ whenever $\phi$ displaces $U$ (here we assume that the empty set is displaced by
 $\id$). Now fix an open ball $U$ in $M$. For any proper leaf $L$ such that $U\cap L$ is an open ball in $L$  the energy-capacity inequality for $L$,  see, e.g., \cite{MS}, p. 377, takes the form
\begin{equation} \frac{1}{2}c_L(U\cap L)\leq E_L(U\cap L), \end{equation}
where $c_L$ is the Gromov width for $L$. Combining  (3.4) and (3.5) we get
\begin{equation*}
\frac{1}{2}c_L(U\cap L)\leq E(U).
\end{equation*}
Thus the inequality (1.6) holds. 

To show the second assertion, note that for an open subset $U$ of $M$ there is, by assumption, a proper leaf $L$ such that $U\cap L\neq\emptyset$. It follows the existence of an open ball
$V\s U$ such that $V\cap L$ is an open ball in $L$. Therefore, in view of (1.6) for $V$, one has
$$
0<\frac{1}{2}c_L(V\cap L)\leq\frac{1}{2}c_{\La}(V)\leq E(V)\leq E(U),$$
as required.

\section{Integrability of Poisson manifolds}

\subsection{Symplectic groupoids and the integrability of Poisson manifolds}
Let $\gg=(\gg\rr M, \ss, \tt, \mm, \ii)$ be a Lie groupoid, where $M$ is the space of units, $\ss, \tt:\gg\to M$ are the source and target maps, $\mm:\gg_2\to \gg$ is the multiplication, where
$\gg_2=\{(x,y)\in \gg\t \gg:\ss(x)=\tt(y)\}$,   and  $\ii:\gg\to \gg$ is the inversion. Here $M$ as well as the fibers of $\ss$ and $\tt$ are
Hausdorff while $\gg$ is possibly non-Hausdorff. This is motivated by important examples, e.g. that of the fundamental groupoid of some kind of foliations (cf. \cite{CDW}, \cite{CF3}).

A Lie algebroid is called \emph{integrable} if it is the Lie algebroid of a Lie groupoid.
For any Lie algebroid $A$, Crainic and Fernandes constructed in \cite{CF1}  a topological groupoid $\mathcal G(A)$ and proved a criterion for the integrability of Lie algebroids.
This criterion says that $A$ is integrable if and only if the groupoid $\mathcal G(A)$ is a Lie groupoid. Equivalently,  $A$ is integrable if and only if the monodromy groups $\mathcal N_x(A)$,
where $x\in M$, are locally uniformly discrete. In this case, $\mathcal G(A)$ is the unique $\ss$-simply connected Lie groupoid integrating $A$.   

\begin{dff}
 A Lie groupoid $\gg\rr M$ equipped with a symplectic
form $\o$ is called \emph{symplectic} if the graph of multiplication
$\graph(\mm)$ is a Lagrangian submanifold of $(-\gg)\t\gg\t\gg$.
Here $-\gg$ denotes the symplectic manifold $(\gg,-\o)$. Similarly as for the Lie groupoids it is assumed that $M$ and the fibers of $\ss$ and $\tt$ are Hausdorff, but
$\gg$ itself is possibly non-Hausdorff. 
\end{dff}

The non-Hausdorff assumption is motivated by important examples of non-Hausdorff symplectic groupoids (cf. \cite{CF2}). However,  in our Theorem 1.3 we need the Hausdorff integrability, and all symplectic groupoids
which appear in examples in the present paper are Hausdorff. The reason that we need $\gg$ to be Hausdorff is the following. Although, in the case of a non-Hausdorf $\gg$,  the lifted Hamiltonian $\ss^*F$, where $F\in\ccmm$, 
in the proof of Theorem 1.3 (below) yields a well-defined flow which lies on the $\tt$-fibers, the modified Hamiltonian $F^{\lambda}$ on $\gg$ possibly does not.
\begin{prop} (cf. \cite{We1}, \cite{CDW}, \cite{CF2})
Let $(\gg, \omega)$ be a symplectic groupoid. Then we have:

\begin{enumerate}
\item The inversion $\ii$ is an antisymplectomorphism (i.e. $\ii^*\o=-\o$),
and $M$ is a Lagrangian submanifold of $\gg$.

\item The foliations by fibers of $\ss$ and of $\tt$ are $\o$-orthogonal.

\item The smooth functions on $\gg$ constant on $\ss$-fibers  and  the smooth functions on $\gg$ constant on $\tt$-fibers  commute.

\item   If $\gg$ is $\ss$-connected the   symplectic foliation $\F_{\La}$ on $M$ coincides with the foliation on $M$ induced by the groupoid structure.

\item The space of units $M$ admits a natural Poisson structure $\La$ such that  $\ss$ (resp. $\tt$) is a Poisson morphism (resp. anti-morphism) of $(\gg,\o)$ onto $(M,\La)$.

\item The Lie algebroid of $\gg$ is canonically isomorphic with the Lie algebroid $(T^*M, [\cdot,\cdot], \sharp_{\La})$ induced by $\m$, where $[\cdot,\cdot]$ is the bracket
 of 1-forms  induced by the Poisson structure $\La$.
 \end{enumerate}
\end{prop}

The basic examples of symplectic groupoids are the following.
\begin{exa}
\begin{enumerate}

\item Given a symplectic manifold $(M,\o)$, the pair groupoid $\gg=M\t M$  with the symplectic form $(-\o)\oplus\o$ is a
symplectic groupoid. Then the group of (global) bisections of $\gg$ coincides with the symplectomorphism group $\symp(M,\o)$.

\item If $M$ is a manifold then $T^*M$, where the multiplication $\mm$ is the addition in
fibers and $\pi_M=\ss=\tt$, is a Lie groupoid with the fiberwise multiplication. $T^*M$ endowed
with the canonical symplectic form $\o_M=d\la_M$ is also a symplectic
groupoid. In fact, the graph of $\mm$
$$
\graph(\mm)=\{(x_3,x_2,x_1):\, x_1+x_2-x_3=0\}.
$$
 is the image of $\mathcal N\Delta_M$, the normal bundle  of the diagonal
$\Delta_M\s M^3$ in $(T^*M)^3$, into $(T^*M)^3$ by the mapping
$(x_3,x_2,x_1)\mt(-x_3,x_2,x_1)$. Notice that $\mathcal N\Delta_M$ is
Lagrangian in $(T^*M)^3$, and the mapping is symplectic. So $T^*M$
is indeed a symplectic groupoid.

\item Let $\gg\rightrightarrows M$ be a Lie groupoid. First observe that the tangent space  $T\gg=(T\gg
\rightrightarrows TM, T\ss, T\tt, \oplus, I)$
carries a structure of Lie groupoid. Here the multiplication $\oplus$
is given by
$$
X\oplus Y=(\ddt(x(t).y(t))|_{t=0},
$$
where $X={dx\over dt}|_{t=0}$, $Y={d y\over d t}|_{t=0}$, $\ss(x(t))=\tt(y(t))$,
and the inversion $IX={d x(t)^{-1}\over dt}|_{t=0}$ if $X={dx\over dt}|_{t=0}$.

Next the cotangent space $T^*\gg$ equipped with 
 $\o_{\gg}=d\lambda_{\gg}$ carries a structure of symplectic groupoid with $\mathcal N^*M$,
the conormal bundle of $M$ in $\gg$, being the space of units.
Here the multiplication, denoted also by $\oplus$, is determined by the equality
$$
\langle\xi\oplus\eta,X\oplus Y\rangle=\langle\xi,X\rangle+\langle\eta,Y\rangle,\quad\hbox{for}\quad X,Y\in T\gg,\,\xi, \eta\in T^*\gg,
$$
where $\langle\cdot,\cdot\rangle$ is the canonical pairing.
Furthermore, the canonical projection $p:T^*\gg\r\gg$ is an epimorphism
of groupoids. Obviously, if $\gg$ is Hausdorff then so are $T\gg$ and $T^*\gg$.

\item If $G$ is a Lie group, $T^*G$ admits two symplectic groupoid
structures. The first one is given as above, and the second is the
structure of the transformation groupoid, where $G$ acts on $\frak g^*$
by the coadjoint action ($\frak g$ is the Lie algebra of $G$). Since these structures fulfill a compatibility
condition, $T^*G$ carries a structure of \emph{ double groupoid}, cf.\cite{CDW}, \cite{Mk}.
\end{enumerate}
\end{exa}

\begin{dff}
A Poisson manifold is called \emph{integrable} (resp. \emph{Hausdorff integrable}) if it can be represented as the space of units of a (resp. Hausdorff) symplectic groupoid. 
\end{dff}

By using their own integrability criterion for Lie algebroids, mentioned above, Crainic and Fernandes proved in \cite{CF2} an integrability criterion for Poisson manifolds.
Namely, they proved the following 

\begin{thm} (\cite{CF2}, Theorem 2; \cite{CF3}, Theorem 5.15) For any Poisson manifold $\m$ the following conditions are equivalent:
\begin{enumerate}

\item $\m$ is  integrable.
\item  The  Lie algebroid $(T^*M, [\cdot,\cdot], \sharp_{\La})$ is integrable.
\item The groupoid $\Sigma(M)$ is a  Lie groupoid, where  $\Sigma(M):=\mathcal G(T^*M)$.
\item The monodromy groups $\mathcal N_x(T^*M)$, with $x\in M$, are locally uniformly discrete.
\end{enumerate}
In this case, $\Sigma(M)$ is the unique $\ss$-simply connected Lie groupoid which integrates $\m$.   
\end{thm}

However, this criterion does not concern the Hausdorff integrability and it seems that it would be difficult to extract a possible characterization of the
 Hausdorff integrability from the proof of Theorem 2 in \cite{CF2}.

A \emph{symplectic realization} of  a Poisson manifold  $(M,\La)$ is a Poisson morphism $\mu:\Sigma\to M$, where $\Sigma$ is a symplectic manifold, such that $\mu$ is a surjective submersion. Similarly as in Definition 4.1, $\Sigma$ is possibly non-Hausdorff, but the leaves of the two foliations on $\Sigma$ induced by $\mu$, that is $\mu^{-1}(x)$ and $\mu^{-1}(x)^\perp$ for $x\in M$, are assumed to be Hausdorff.
In particular, the notion of the completeness of the vector field $X_{\mu^*H}$, where $H\in\cm$, makes sense. Next,
a symplectic realization $\mu:\Sigma\to M$ is called
\emph{complete} if for any complete Hamiltonian vector field $X_H$ on $M$ with $H\in\cm$, the vector field $X_{\mu*H}$ is complete.

A classical result due to Karasev and Weinstein (cf. \cite{CF2}, Theorem 7) states that any Poisson manifold admits a Hausdorff symplectic realization. Since the paper by Crainic and Fernandes \cite{CF2}
we know that the concept of integrability of Poisson manifolds can be expressed in terms of possessing a complete symplectic realization.
Namely, Theorem 8 in \cite{CF2} says that a  Poisson manifold is integrable if and only if it admits a complete symplectic realization. It is straightforward from
an easier part of the proof of this theorem that the following holds.

\begin{prop}
If $(\gg\rr M, \ss, \tt)$ is a symplectic groupoid then its source map $\ss:\gg\to M$ is a complete symplectic realization of $M$.
\end{prop}

As an obvious consequence of the proof of Theorem 1.3 below we get
\begin{cor}
If a Poisson manifold $M$ admits a complete Hausdorff symplectic realization $\mu:\Sigma\to M$  then $\rh$ is non-degenerate. In particular, this is the case whenever $\Sigma$
is compact and Hausdorff.
\end{cor}
In fact, it suffices to replace $\ss$ by $\mu$.

\subsection{Proof of Theorem 1.3} Let $\m$ be integrated by a Hausdorff symplectic groupoid $(\gg,\o)$.
 In view of \cite{Mk} we may and do assume that  $\gg$ is $\ss$-connected.

Suppose $\phi\in\ham\m\setminus\{\id\}$ and let $F\mt\phi$ with $F\in\ccmm$, i.e. $\Phi_F=(\{\phi_F^t\},F)$ and $\phi=\phi_F^1$. Denote by $\tilde X_F$ the Hamiltonian vector field on  $\gg$
generated by $\ss^*F=F\circ\ss$. According to Proposition 4.6 the isotopy $\{\tilde\phi_F^t\}$ of $\tilde X_F^t$ is defined for all $t\in\R$. The isotopy
$\Phi_{\ss^*F}=(\{\tilde\phi_F^t\},\ss^*F)$ is Hamiltonian but not necessarily compactly supported. However its length is well defined by (1.4) and
\begin{equation} \l(\Phi_{\ss^*F})=\l(\Phi_F).
\end{equation}  Furthermore, in view of (2.8) we have $\ss_*\tilde X_F=X_F$ and we get that the
isotopy $\tilde\phi_F^t$ projects to $\phi_F^t$, that is  for all $t$ \begin{equation}\ss\ci\tilde\phi_F^t=\phi_F^t\ci\ss.\end{equation}

Now fix an open subset $U\s M$ such that $\phi(U)\cap U=\emptyset$. In view of (4.2) we have $\tilde\phi^1_F(\tilde U)\cap \tilde U=\emptyset$, where  $\tilde U=\ss^{-1}(U)$. Let $\lambda:\gg\to[0,1]$ be a smooth compactly supported cut-off function such that $\lambda=1$ on a neighborhood of a fixed ball $B\s \tilde U$. 
Denote by  $\{\psi^t\}$ the isotopy on $(\gg,\o)$ generated by $F^{\la}\in\cc_c(I\t\gg,\R)$ such that 
$$F^{\la}(t,x)=\lambda((\tilde\phi_F^t)^{-1}(x))F(t,\ss(x)),\quad\forall (t,x)\in I\t\gg.$$  It follows that $\psi=\psi^1$ displaces $B$. Thus, in view
of the energy-capacity inequality $\frac{1}{2}c(B)\leq E(B)$ for $(\gg,\o)$ and (4.1) we obtain
\begin{equation} 0<\frac{1}{2}c(B)\leq E(B)\leq\length(\Phi_{F^\la})\leq\l(\Phi_{\ss^*F})=\l(\Phi_F),
\end{equation}
where $c(B)$ is the Gromov width of $B$ in $\gg$. It follows that $\ra(\phi,\id)>0$, as required.

\section{Examples of integrable Poisson manifolds}

There are several remarkable papers on the integrability of  Poisson manifolds. Here we give some examples of Hausdorff integrable Poisson
manifolds.

\medskip
(1) First we observe that 
the Poisson structure of a symplectic manifold $(M,\o)$ is integrable. In fact, one can integrate it by the pair groupoid $M\t \overline M$ equipped with the
symplectic structure $\o\oplus-\o=\ss^*\o-\tt^*\o$. 
\medskip

(2) Let $\frak g$ be a finite dimensional Lie algebra.  A \emph{Lie-Poisson structure} on $\frak g^*$, the dual of $\frak g$, is integrable according to Example 4.3(4).
Here the leaves of the symplectic foliation on $\frak g^*$ coincide with orbits of the coadjoint action of $G$ on $\frak g$, where $G$ is a connected Lie group with the Lie algebra $\frak g$.

More precisely, there are two left actions of $G$ on itself, namely by the left translations $L_g$ and by the right translations $R_{g^{-1}}$, $g\in G$. The lifts of these actions
to the cotangent bundle $T^*G$ are defined by $\Phi_g=L_{g^{-1}}^*$ and $\Psi_g=R^*_g$, resp. Then $\Phi_g$ and $\Psi_g$ are Hamiltonian actions that have equivariant momentum maps
$J^\Phi,\, J^\Psi: T^*G\to\frak g^*$. Thus we obtain a symplectic groupoid $(T^*G\rr \frak g^*, J^\Psi, J^\Phi)$, where $J^\Psi$ is the source and $J^\Phi$ is the target.
For more details, see, e.g., \cite{Vai}, p.135-143.

\medskip

(3) Let $\m$ be a Poisson manifold. A submanifold $N\s M$ is called a \emph{Lie-Dirac submanifold} (cf. \cite{CF2}, section 9.3) if there is a vector bundle $E$ over $N$
such that $T_NM=TN\oplus E$ and $E^0\s T^*M$ is a subalgebroid. Equivalently, $N$ can be described as a Poisson-Dirac submanifold of $M$ which admits a Dirac projection
$p:T_NM\to TN$ (Proposition 7 in \cite{CF2}).
In view of Theorem 9 in \cite{CF2}  any Lie-Dirac submanifold $N$ of an integrable Poisson manifold $M$, is also integrable, and $\Sigma(N)$ is a symplectic subgroupoid
of $\Sigma(M)$ (cf. Theorem 4.5). Thus any Lie-Dirac submanifold $N$ of an integrable Poisson manifold $M$ with $\Sigma(M)$ Hausdorff admits a Hausdorff integration.

\medskip

(4)   Let $G$ be a Lie group with its Lie algebra $\frak g$ and with a free proper Hamiltonian action on a symplectic manifold $(N,\sigma)$ such that the momentum map $J: N\to\frak g^*$ is equivariant. Then the quotient $N/G$, together with the reduced Poisson structure induced
from $N$, is an integrable Poisson manifold. Indeed, a symplectic groupoid integrating $N/G$ is given by the Marsden-Weinstein reduction  of the pair groupoid
$N\t\overline N$  with respect to the diagonal action of $G$, denoted by $N\star\overline N/G$. That is, $N\star\overline N/G\rr N/G$, where 
$$  N\star\overline N=\{(x,y)\in N\t\overline N|\, J(x)=J(y)\}. $$
 For details, see \cite{DZ}, p.227.

\medskip

(5) Assume  that $\m$ is an integrable Poisson manifold with $\Sigma(M)$ Hausdorff and that $\Psi:G\t M\to M$ is a proper and free Poisson action
of a Lie group $G$. Therefore the lifted
action $\Sigma(\Psi): G\t\Sigma(M)\to \Sigma(M)$ is proper and free as well. It follows that $0\in\frak g^*$ is a regular value of the momentum map $J:\Sigma(M)\to\frak g^*$.
Since $J$ is a $G$-equivariant groupoid homomorphism, its kernel $J^{-1}(0)\s\Sigma(M)$ is a $G$-invariant Lie subgroupoid. It follows that  the symplectic quotient $J^{-1}(0)/G$ carries a Hausdorff symplectic groupoid structure, that is we get a  Hausdorff symplectic groupoid  $J^{-1}(0)/G\rr M/G$.
Consequently, the quotient Poisson manifold $M/G$ is Hausdorff integrable. See \cite{CF3}, p.101, for details.

\end{document}